\documentclass[11pt]{article} 
\usepackage{amsfonts,amsmath,latexsym,amssymb,mathrsfs,amsthm,comment}
\usepackage{graphicx}
\usepackage{xcolor}
\usepackage{booktabs}

\let\OLDthebibliography\thebibliography
\renewcommand\thebibliography[1]{
  \OLDthebibliography{#1}
  \setlength{\parskip}{1pt}
  \setlength{\itemsep}{0pt plus 0.0ex}
}

\def\numberlikeadb{\global\def\theequation{\thesection.\arabic{equation}}}
\numberlikeadb
\newtheorem{theorem}{Theorem}[section]

\newtheorem{corollary}[theorem]{Corollary}

\newtheorem{proposition}[theorem]{Proposition}
\newtheorem{remark}[theorem]{Remark}

\usepackage{lscape}
\usepackage{caption}
\usepackage{multirow}
\allowdisplaybreaks
\begin{document}

\title{On the product of correlated normal random variables and the noncentral chi-square difference distribution
}
\author{Robert E. Gaunt\footnote{Department of Mathematics, The University of Manchester, Oxford Road, Manchester M13 9PL, UK, robert.gaunt@manchester.ac.uk}}

\date{} 
\maketitle

\vspace{-5mm}

\begin{abstract}  We represent the product of two correlated normal random variables, and more generally the sum of independent copies of such random variables, as a difference of two independent noncentral chi-square random variables (which we refer to as the noncentral chi-square difference distribution). As a consequence, we obtain, amongst other results, an exact formula for the probability density function of the noncentral chi-square difference distribution, a Stein characterisation of the noncentral chi-square difference distribution, a simple formula for the moments of the sum of independent copies of the product of correlated normal random variables, an exact formula for the probability that such a random variable is negative, and also show that such random variables are self-decomposable and provide a L\'evy-Khintchine representation of the characteristic function.
\end{abstract}



\section{Introduction}

Let $(X, Y)$ be a bivariate normal random vector with mean vector $(\mu_X,\mu_Y)$, variances $(\sigma_X^2,\sigma_Y^2)$ and correlation coefficient $\rho$. Starting with the work of \cite{craig,wb32} in the 1930's, the distribution of the product $Z=XY$, which we denote by $PN(\mu_X,\mu_Y;\sigma_X^2,\sigma_Y^2;\rho)$, 
has received much attention in the statistics literature (see \cite{gaunt22,np16} for an overview), and has found numerous applications in fields such as chemical physics \cite{hey}, condensed matter physics \cite{ach} and statistical mediation analysis \cite{mac}. The sum $S_n=\sum_{i=1}^nZ_i$, where $Z_1,\ldots,Z_n$ are independent copies of $Z$, also has applications in areas such as electrical engineering \cite{ware}, astrophysics \cite{man} and quantum cosmology \cite{gr}.

The difference of independent chi-square random variables (the chi-square difference distribution) has also received interest in the literature; we refer the reader to \cite{k15} and \cite{j22} for historical remarks and application areas. The difference of independent chi-square random variables is variance-gamma (VG) distributed (see \cite[Section 2.5]{gaunt22}), and so results for the chi-square difference distribution can be inferred from the distributional theory for the VG distribution given in the review \cite{vg review} and Chapter 4 of the book \cite{kkp01}. However, there is a natural gap in the literature in that the basic distributional theory for the difference of independent noncentral chi-square random variables (which we will refer to as the \emph{noncentral chi-square distribution}) has yet to be explored.


In this paper, we obtain a 
representation of the product $Z$, and more generally the sum $S_n$, as a difference of two independent noncentral chi-square random variables. Our result (Theorem \ref{thm1}) generalises the representation of $S_n$ as a difference of two independent chi-square random variables in the zero mean case ($\mu_X=\mu_Y=0$); see \cite[Section 2.5]{gaunt22}. Since $S_n$ is VG distributed in the zero mean case (see \cite{gaunt prod}), this result can be inferred from the representation of the VG distribution as a difference of independent chi-square random variables (see \cite{vg review}).
The connection of Theorem \ref{thm1} between the distribution of the sum $S_n$ and the noncentral chi-square difference distribution is rather useful, as it allows one to immediately deduce results for one of these distributions once results have been established for the other. The utility lies in the fact that for the purpose of deriving distributional properties it is often simpler to work with either the sum $S_n$ or the noncentral chi-sqaure difference distribution.

As applications of Theorem \ref{thm1}, we establish a number of key distribution properties for the noncentral chi-square difference distribution and for the distribution of the sum $S_n$. For the noncentral chi-square difference distribution, we obtain, amongst other results, an exact formula for the probability density  function (PDF) (Corollary \ref{corad}), a Stein characterisation (Corollary \ref{corstein}) and formulas for the moments and cumulants (Proposition \ref{prop1}). For the sum $S_n$, and hence the product $Z$ as a special case on setting $n=1$, we show that the distribution is self-decomposable and provide a L\'evy-Khintchine representation of the characteristic function (Corollary \ref{cor444}), as well as formulas for the moments and cumulants (Corollary \ref{thm2}) and an exact formula for the probability that the sum $S_n$ is negative (Corollary \ref{corlast}).

\section{Results and proofs}

In the following theorem, we represent the sum $S_n$ as a difference of independent noncentral chi-square random variables. Before stating the result, we recall that the sum $\sum_{i=1}^r X_i^2$, where $X_i\sim N(\mu_i,1)$, $i=1,\ldots,r$, are independent normal random variables, follows the noncentral chi-square distribution $\chi_r'^2(\lambda)$ with $r>0$ degrees of freedom and noncentrality parameter $\lambda=\sum_{i=1}^r\mu_i^2\geq0$. The chi-square distribution with $r$ degrees of freedom corresponds to the case $\lambda=0$. For $\lambda>0$, the PDF is given by
\begin{equation}\label{pdf0}
p(x)=\frac{1}{2}\mathrm{e}^{-(x+\lambda)/2}\bigg(\frac{x}{\lambda}\bigg)^{r/4-1/2}I_{r/2-1}(\sqrt{\lambda x}),\quad x\geq0,   
\end{equation}
where $I_\nu(x)$ is a modified Bessel function of the first kind (see \cite[Chapter 10]{olver}). 

\begin{theorem}\label{thm1} 1. Let $\mu_X,\mu_Y\in\mathbb{R}$, $\sigma_X,\sigma_Y>0$, $\rho\in(-1,1)$ and $n\geq1$. Let $s=\sigma_X\sigma_Y$. Then
\begin{equation}\label{eq0}
S_n=_d \frac{s}{2}(1+\rho)V_1-\frac{s}{2}(1-\rho)V_2,   
\end{equation}
where $V_1\sim\chi_n'^2(\lambda_+)$ and $V_2\sim\chi_n'^2(\lambda_-)$ are independent with
\begin{equation}\label{lambda}
\lambda_+=\frac{n}{2(1+\rho)}\bigg(\frac{\mu_X}{\sigma_X}+\frac{\mu_Y}{\sigma_Y}\bigg)^2,\quad \lambda_-=\frac{n}{2(1-\rho)}\bigg(\frac{\mu_X}{\sigma_X}-\frac{\mu_Y}{\sigma_Y}\bigg)^2.    
\end{equation}
\noindent 2. Consider now the degenerate case $\rho=1$. Then
\begin{equation*}
S_n=_dsV_1-\frac{ns}{4}\bigg(\frac{\mu_X}{\sigma_X}-\frac{\mu_Y}{\sigma_Y}\bigg)^2,    
\end{equation*}
where $V_1\sim\chi_n'^2(\lambda_+)$ with $\lambda_+=n(\mu_X/\sigma_X+\mu_Y/\sigma_Y)^2/4$. Similarly, if $\rho=-1$ we have that
\begin{equation*}
S_n=_d\frac{ns}{4}\bigg(\frac{\mu_X}{\sigma_X}+\frac{\mu_Y}{\sigma_Y}\bigg)^2-sV_2,    
\end{equation*}
where $V_2\sim\chi_n'^2(\lambda_-)$ with $\lambda_-=n(\mu_X/\sigma_X-\mu_Y/\sigma_Y)^2/4$.
\end{theorem}

\begin{remark} 
In the cases $\mu_X/\sigma_X+\mu_Y/\sigma_Y=0$ and $\mu_Y/\sigma_X-\mu_Y/\sigma_Y=0$, the random variables $V_1$ and $V_2$, respectively, follow the chi-square distribution with $n$ degrees of freedom. In particular, in the case $\mu_X=\mu_Y=0$, we recover the known representation of the sum $S_n$ as a difference of independent chi-square random variables with $n$ degrees of freedom (see \cite[Section 2.5]{gaunt22}).
\end{remark}

\begin{proof} 1. For ease of notation, we shall prove the result for the case $\sigma_X=\sigma_Y=1$, with the general case $\sigma_X,\sigma_Y>0$ following from the basic relation that $Z=XY=_d \sigma_X\sigma_Y N_1N_2$, where $(N_1,N_2)$ is a bivariate normal random vector with mean vector $(\mu_X/\sigma_X,\mu_Y/\sigma_Y)$, variances $(1,1)$ and correlation coefficient $\rho$.

 By \cite[equation (10)]{craig}, we have that the characteristic function of $Z$ is given by 
 \begin{align*}
      \varphi_Z(t) = \frac1{([1-(1+\rho)\mathrm{i}t][1+(1-\rho)\mathrm{i}t])^{1/2}} 
			             \exp\bigg(\frac{-(\mu_X^2+\mu_Y^2-2\rho\mu_X\mu_Y)t^2+2\mu_X\mu_Y\mathrm{i}t}
									 {2[1-(1+\rho)\mathrm{i}t][1+(1-\rho)\mathrm{i}t]}\bigg),
   \end{align*}
for $t\in\mathbb{R}$. Then, by the standard formula for the characteristic function of sums 
of independent random variables,   
\begin{align*}
\varphi_{S_n}(t)=\frac{1}{([1-(1+\rho)\mathrm{i}t][1+(1-\rho)\mathrm{i}t])^{n/2}}\exp\bigg(\frac{-n(\mu_X^2+\mu_Y^2-2\rho\mu_X\mu_Y)t^2+2n\mu_X\mu_Y\mathrm{i}t}{2[1-(1+\rho)\mathrm{i}t][1+(1-\rho)\mathrm{i}t]}\bigg).
\end{align*}
A simple manipulation now shows that
\begin{align}
\varphi_{S_n}(t)&=\frac{1}{(1-(1+\rho)\mathrm{i}t)^{n/2}}\exp\bigg(\frac{\mathrm{i}(1+\rho)\lambda_+ t}{2(1-(1+\rho)\mathrm{i}t)}\bigg)\nonumber\\
&\quad\times \frac{1}{(1+(1-\rho)\mathrm{i}t)^{n/2}}\exp\bigg(-\frac{\mathrm{i}(1-\rho)\lambda_- t}{2(1+(1-\rho)\mathrm{i}t)}\bigg)\nonumber\\
&=\varphi_{U_1}(t)\varphi_{-U_2}(t)=\varphi_{U_1-U_2}(t), \label{eq1}
\end{align}
where $U_1=(1+\rho)V_1/2$ and $U_2=(1-\rho)V_2/2$, and we used the fact that the characteristic function of the $\chi_r'^2(\lambda)$ distribution is given by 
\[\varphi(t)=\frac{1}{(1-2\mathrm{i}t)^{r/2}}\exp\bigg(\frac{\mathrm{i}\lambda t}{1-2\mathrm{i}t}\bigg), \quad t\in\mathbb{R}\]
(see \cite{p49}). The theorem now follows from (\ref{eq1}) and the uniqueness of characteristic functions. 

\vspace{2mm}

\noindent 2. Again, we set $\sigma_X=\sigma_Y=1$. We consider the case $\rho=1$; the case $\rho=-1$ is similar and is omitted. For $\rho=1$, we have that $Z=XY=_d (N+\mu_X)(N+\mu_Y)=(N+(\mu_X+\mu_Y)/2)^2-(\mu_X-\mu_Y)^2/4$, where $N\sim N(0,1)$. Therefore $S_n=\sum_{i=1}^nZ_i=_d\sum_{i=1}^n U_i^2 -n(\mu_X-\mu_Y)^2/4$, where $U_i\sim N((\mu_X+\mu_Y)/2,1)$, $i=1,\ldots,n$. The result now follows since $\sum_{i=1}^nU_i^2\sim \chi_n'^2(\lambda)$, where $\lambda=\sum_{i=1}^n(\mu_X+\mu_Y)^2/4=n(\mu_X+\mu_Y)^2/4$.
\end{proof}

An exact formula for the PDF of the sum $S_n$ was recently obtained by \cite{gnp24}; an exact formula for the product $Z$ had previously been obtained by \cite{cui}. Combining the formula of \cite{gnp24} and the representation (\ref{eq0}) we obtain the following exact formula for the PDF of the difference of two independent noncentral chi-square random variables. The formula is expressed in terms of the confluent hypergeometric function of the second kind $U(a,b,x)$ (see \cite[Chapter 13]{olver}). In the case $\lambda_1=\lambda_2$, a simpler formula for the PDF is also given in terms of the modified Bessel function of the second $K_\nu(x)$ (see \cite[Chapter 10]{olver}).

\begin{corollary}\label{corad} Let $T=V_1-V_2$, where $V_1\sim\chi_r'^2(\lambda_1)$ and $V_2\sim\chi_r'^2(\lambda_2)$ are independent.

\vspace{2mm}

\noindent 1. Let $r>0$ and $\lambda_1,\lambda_2\geq0$. Then
\begin{align}
p_T(x)&=\frac{1}{2^{r}}\mathrm{e}^{-(|x|+\lambda_1+\lambda_2)/2}\sum_{k=0}^\infty\sum_{j=0}^k\frac{1}{k!\Gamma(r/2+a_{j,k}(x))}\binom{k}{j} \bigg(\frac{\lambda_1}{4}\bigg)^{k-j}\bigg(\frac{\lambda_2}{4}\bigg)^j \nonumber\\
&\quad\times  U\bigg(1-\frac{r}{2}-a_{j,k}(x),2-r-k,|x|\bigg), \quad x\in\mathbb{R},  \label{inf2}
\end{align}
where $a_{j,k}(x)=k-j$ if $x\geq0$ and $a_{j,k}(x)=j$ if $x<0$.  

\vspace{2mm}

\noindent 2. Suppose now that $\lambda_1=\lambda_2=\lambda$. Then, the PDF of $T$ can be expressed as a single infinite series:
\begin{equation}\label{rdhl}
p_T(x)=\frac{1}{2^r\sqrt{\pi}}\mathrm{e}^{-\lambda}\sum_{k=0}^\infty\frac{(\lambda/4)^k}{k!\Gamma(r/2+k)}|x|^{(r-1)/2+k}K_{\frac{r-1}{2}+k}\bigg(\frac{|x|}{2} \bigg),\quad x\in\mathbb{R}.    
\end{equation}
\end{corollary}

\begin{proof}1. By Theorem \ref{thm1}, $T=_d 2\sum_{i=1}^rZ_i$, where the $Z_i\sim PN(\mu_X,\mu_Y;1,1;0)$ are independent, and $\mu_X$ and $\mu_Y$ satisfy $(\mu_X+\mu_Y)^2=2\lambda_1/r$ and $(\mu_X-\mu_Y)^2=2\lambda_2/r$. The result now follows from applying Theorem 2.1 of \cite{gnp24} (with appropriate parameter values), which provides an exact formula for the PDF of the sum  $\sum_{i=1}^rZ_i$, and thus an exact formula for the PDF of $2\sum_{i=1}^rZ_i$ after a simple re-scaling.
	
\vspace{2mm}

\noindent 2. This is similar to part 1, except that we instead apply Theorem 2.6 of \cite{gnp24}, which gives a formula for the PDF of the sum $S_n$ in the case that $\mu_Y=\rho=0$ (in which case $\lambda_+=\lambda_-$).	
\end{proof}

\begin{remark} When $\lambda_1=0$ or $\lambda_2=0$, the PDF (\ref{inf2}) reduces to a single infinite series. For example, with $\lambda_2=0$ we get, for $x\in\mathbb{R}$,
\begin{align*}
p_T(x)&=\frac{1}{2^{r}}\mathrm{e}^{-(|x|+\lambda_1)/2}\sum_{k=0}^\infty\frac{(\lambda_1/4)^k}{k!\Gamma(r/2+a_{0,k}(x))}   U\bigg(1-\frac{r}{2}-a_{0,k}(x),2-r-k,|x|\bigg).
\end{align*}
When $\lambda_1=\lambda_2=0$, the PDF (\ref{inf2}) reduces to a single term:
\begin{equation}\label{rdh}
p_T(x)= \frac{1}{2^{r}\Gamma(r/2)}\mathrm{e}^{-|x|/2}U\bigg(1-\frac{r}{2},2-r,|x|\bigg), \quad x\in\mathbb{R}.   
\end{equation}
By using the formulas $U(a,2a,2x)=\pi^{-1/2}\mathrm{e}^{x}(2x)^{1/2-a}K_{a-\frac{1}{2}}(x)$ (see \cite[equation 13.6.10]{olver}) and $K_a(x)=K_{-a}(x)$ (see \cite[equation 10.27.3]{olver}) we obtain that
\begin{equation}\label{rdh2}
p_T(x)=\frac{1}{2^r\sqrt{\pi}\Gamma(r/2)}|x|^{(r-1)/2} K_{\frac{r-1}{2}}\bigg(\frac{|x|}{2}\bigg), \quad x\in\mathbb{R}, 
\end{equation}
which is the PDF of a symmetric VG random variable (see \cite{vg review}). Setting $\lambda=0$ (so that $\lambda_1=\lambda_2=0$) in (\ref{rdhl}) also yields (\ref{rdh2}).
\end{remark}

\begin{corollary}\label{cor38}
Let $r>0$, $\lambda_1,\lambda_2\geq0$ and let the random variable $T$ be defined as in Corollary \ref{corad}. 

\vspace{2mm}



\noindent 1. The distribution of $T$ is unimodal.

\vspace{2mm}

\noindent 2. Moreover, the PDF (\ref{inf2}) is bounded for all $x\in\mathbb{R}$ if and only if $r>1$, and in the case $r=1$ the distribution of $T$ has mode 0, with the singularity possessing the following asymptotic behaviour:
\begin{align*}p_T(x)\sim -\frac{1}{2\pi}\mathrm{e}^{-(\lambda_1+\lambda_2)/2}\ln|x|, \quad x\rightarrow0.
\end{align*}
\end{corollary}

\begin{proof}
\noindent 1. Since $T$ is a difference of independent noncentral chi-square random variables, and the noncentral chi-square distribution is self-decomposable (see \cite{j21}), it follows that the distribution of $T$ is self-decomposable. As $T$ is self-decomposable, it follows that $T$ is unimodal, since self-decomposable distributions are unimodal (see \cite{y78}).

\vspace{2mm}

\noindent 2. This can be inferred from the representation (\ref{eq0}) and the corresponding results of \cite[Proposition 2.1]{gz23} and \cite[Corollary 2.2]{gnp24} for the sum $S_n$. Alternatively, part 3 can be deduced from the fact that the function $U(a,b,x)$ is bounded for all non-zero finite values of $x\in\mathbb{R}$, and has the following limiting forms: $U(a,1,x)\sim-\ln|x|/\Gamma(a)$, as $x\rightarrow0$, and $U(a,b,x)\sim\Gamma(1-b)/\Gamma(a-b+1)$, as $x\rightarrow0$, (for $b<1$ and $a-b+1>0$) (see \cite{olver}).
\end{proof}

We are also able to exploit the connection (\ref{eq0}) between the difference of independent noncentral chi-square random variables and the product of correlated normal random variables in order to obtain the following Stein characterisation of the noncentral chi-square difference distribution. The result complements a recent Stein characterisation of \cite{f23} for the gamma difference distribution, which is a special case of the Stein characterisation of \cite{gaunt vg} for the VG distribution and the Stein characterisation of \cite{aaps20} for a linear combination of independent gamma random variables.

We let $\mathcal{F}_m$ be the class of functions $f:\mathbb{R}\rightarrow\mathbb{R}$ such that $f\in C^m(\mathbb{R})$ and $\mathbb{E}|f^{(j)}(T)|$, for $0\leq j\leq m-1$, and $\mathbb{E}|Tf^{(j)}(T)|$, for $0\leq j\leq m$, are finite, where $T=V_1-V_2$, with $V_1\sim\chi_r'^2(\lambda_1)$ and $V_2\sim\chi_r'^2(\lambda_2)$ independent. Here $f^{(0)}\equiv f$.

\begin{corollary}\label{corstein} Let $T=V_1-V_2$, where $V_1\sim\chi_r'^2(\lambda_1)$ and $V_2\sim\chi_r'^2(\lambda_2)$ are independent, with $r>0$ and $\lambda_1,\lambda_2\geq0$. Let $W$ be a real-valued random variable such that $\mathbb{E}|W|<\infty$. 

\vspace{2mm}

\noindent 1. Define the operator $A_1$ by
\begin{align}
		A_1f(x)&=16xf^{(4)}(x)+16rf^{(3)}(x)-\big(8x+4(\lambda_1-\lambda_2)\big)f''(x)\nonumber\\
		&\quad-4(\lambda_1+\lambda_2+r)f'(x)+\big(x-(\lambda_1-\lambda_2)\big)f(x).\nonumber
	\end{align}
Then $W=_dT$ if and only if $
		\mathbb{E}[A_{1}f(W)]=0\nonumber
	$
for all $f\in\mathcal{F}_4$.

\vspace{2mm}

\noindent 2. Lower order characterising operators are available if $\lambda_1=0$ or $\lambda_2=0$. Suppose now that $\lambda_2=0$ (the case $\lambda_1=0$ is similar). Define the operator $A_2$ by
\begin{align}
		A_{2}f(x)&=8xf^{(3)}(x)+(8r-4x)f''(x)-(2x+4r+2\lambda_1)f'(x)+(x-\lambda_1)f(x).\nonumber
	\end{align}
Then $W=_dT$ if and only if $
		\mathbb{E}[A_{2}f(W)]=0\nonumber
	$
for all $f\in\mathcal{F}_3$.

\vspace{2mm}

\noindent 3. Suppose that $\lambda_1=\lambda_2=0$. Define the operator $A_3$ by
\begin{align}
		A_{3}f(x)&=4xf''(x)+4rf'(x)-xf(x).\nonumber
	\end{align}
Then $W=_dT$ if and only if $
		\mathbb{E}[A_{3}f(W)]=0\nonumber
	$
for all $f\in\mathcal{F}_2$.
\end{corollary}

\begin{proof}
Parts 1 and 2 are immediate (after a simple re-scaling) from the representation (\ref{eq0}) and the Stein characterisations of the sample mean $S_n/n$ given by Theorems 2.1 and 2.3, respectively, of \cite{gls24}. Part 3 is a special case of the Stein characterisation of \cite{f23} for the gamma difference distribution.
\end{proof}

\begin{remark}
The classical characterising Stein operator for the chi-square distribution with $r$ degrees of freedom is given by $Af(x)=2xf'(x)+(r-x)f(x)$ (see \cite{dz}), which is a first order differential operator. However, the Stein operator of \cite{gaunt34} for the noncentral chi-square distribution $\chi_r'^2(\lambda)$, as given by $Af(x)=4xf''(x)+(2r-4x)f'(x)+(x-r-\lambda)f(x)$, is a second order differential operator. It is therefore quite intuitive that characterising Stein operators of lower order are available for the noncentral chi-square difference distribution if $\lambda_1=0$ or $\lambda_2=0$.
\end{remark}

We now obtain formulas for the moments of the noncentral chi-square difference distribution. These formulas are easily obtained using known formulas for the moments of the noncentral chi-square distribution. We therefore do not apply our representation (\ref{eq0}) to obtain these results; however, given that the moments of a distribution are a key distributional property, and in this case are easily derived, we state these formulas as it will be helpful for researchers to have these useful formulas stated in the literature. 
To this end, we recall that the moments of a noncentral chi-square random variable $V\sim\chi_r'^2(\lambda)$ are given by (see \cite{mathworld1})
\begin{equation}\label{momf}
\mathbb{E}[V^m]=2^m\mathrm{e}^{-\lambda/2}\frac{\Gamma(m+r/2)}{\Gamma(r/2)}M\bigg(m+\frac{r}{2},\frac{r}{2},\frac{\lambda}{2}\bigg), \quad m\geq0,   
\end{equation}
where $M(a,b,x)={}_1F_1(a,b,x)$ is a confluent hypergeometric function of the first kind (see \cite[Chapter 13]{olver}). The formula can be obtained by using \cite[equation 6.643]{g07} to compute the integral $\int_0^\infty x^mp(x)\,\mathrm{d}x$, where the PDF $p$ is given by (\ref{pdf0}). In particular, the first four raw moments are given by
\begin{align*}
\mu_1'&=r+\lambda, \quad \mu_2'=(r+\lambda)^2+2(r+2\lambda), \quad \mu_3'=(r+\lambda)^3+6(r+\lambda)(r+2\lambda)+8(r+3\lambda),\\
\mu_4'&=(r+\lambda)^4+12(r+\lambda)^2(r+2\lambda)+4(11r^2+44r\lambda+36\lambda^2)+48(r+4\lambda),
\end{align*}
and the central moments are
\begin{align*}
\mu_2=2(r+2\lambda), \quad \mu_3=8(r+3\lambda), \quad \mu_4=12(r+2\lambda)^2+48(r+4\lambda).    
\end{align*} 
With the formula (\ref{momf}), the above formulas for the first four raw and central moment, and the representation (\ref{eq0}) we are able to obtain the following formulas for the moments of the noncentral chi-square difference distribution.


\begin{proposition}\label{prop1} Let $T=V_1-V_2$, where $V_1\sim\chi_r'^2(\lambda_1)$ and $V_2\sim\chi_r'^2(\lambda_2)$ are independent, with $r>0$ and $\lambda_1,\lambda_2\geq0$. Then, for $k\geq1$,
\begin{align}
\mathbb{E}[T^k]&=\frac{2^k\mathrm{e}^{-(\lambda_1+\lambda_2)/2}}{(\Gamma(r/2))^2}\sum_{j=0}^k(-1)^{k-j}\binom{k}{j}\Gamma\bigg(j+\frac{r}{2}\bigg)\Gamma\bigg(k-j+\frac{r}{2}\bigg)\nonumber\\
&\quad\times M\bigg(j+\frac{r}{2},\frac{r}{2},\frac{\lambda_1}{2}\bigg)M\bigg(k-j+\frac{r}{2},\frac{r}{2},\frac{\lambda_2}{2}\bigg).\label{momf2}
\end{align} 
In particular, the first four raw moments are given by
\begin{align*}
\mu_1'&=\lambda_1-\lambda_2, \\ 
\mu_2'&=4(r+\lambda_1+\lambda_2)+(\lambda_1-\lambda_2)^2,\\
\mu_3'&=12(\lambda_1-\lambda_2)(r+\lambda_1+\lambda_2+2)+(\lambda_1-\lambda_2)^3,\\
\mu_4'&=48(r+\lambda_1+\lambda_2)^2+96(r+2\lambda_1+2\lambda_2)+24(\lambda_1-\lambda_2)^2(r+\lambda_1+\lambda_2+4)+(\lambda_1-\lambda_2)^4,
\end{align*}
and the central moments are
\begin{align*}
\mu_2&=4(r+\lambda_1+\lambda_2), \\
\mu_3&=24(\lambda_1-\lambda_2),\\
\mu_4&=48(r+\lambda_1+\lambda_2)^2+96(r+2\lambda_1+2\lambda_2).
\end{align*}
The variance $\sigma^2$, skewness $\gamma_1=\mu_3/\sigma^{3}$ and excess kurtosis $\gamma_2=\mu_4/\sigma^4-3$ are given by
\begin{align*}
\sigma^2=4(r+\lambda_1+\lambda_2),\quad \gamma_1=\frac{3(\lambda_1-\lambda_2)}{(r+\lambda_1+\lambda_2)^{3/2}},\quad\gamma_2=\frac{6(r+2\lambda_1+2\lambda_2)}{(r+\lambda_1+\lambda_2)^2}.    
\end{align*}
The cumulants are given by
\begin{align}
\kappa_k=2^{k-1}(k-1)![(r+k\lambda_1)+(-1)^k(r+k\lambda_2)], \quad k\geq1. \label{cuml}
\end{align}
\end{proposition}

\begin{proof}We have that $\mathbb{E}[T^k]=\mathbb{E}[(V_1-V_2)^k]=\sum_{j=0}^k(-1)^{k-j}\binom{k}{j}\mathbb{E}[V_1^j]\mathbb{E}[V_2^{k-j}]$, and applying the moment formula (\ref{momf}) now yields formula (\ref{momf2}), and the formulas for the first four raw and central moments of $T$ are obtained similarly using the formulas for the first four raw and central moments of the noncentral chi-square distribution. The formulas for the skewness and kurtosis of $T$ follow from simple manipulations. Finally, formula (\ref{cuml}) for the cumulants follows from the standard results that, for independent random variables $U_1$ and $U_2$, and constant $c$, the $k$-th cumulant satisfies $\kappa_k(cU_1)=c^k\kappa_k(U_1)$ and $\kappa_k(U_1+U_2)=\kappa_k(U_1)+\kappa_k(U_2)$, together with the fact that cumulants of the noncentral chi-square distribution $\chi_r'^2(\lambda)$ are given by $\kappa_k=2^{k-1}(k-1)!(r+k\lambda)$, $k\geq1$ (see \cite{p49}).
\end{proof}

We now turn our attention to exploiting the relation (\ref{eq0}) to obtain new results for the sum $S_n$; results for the product $Z$ follow on setting $n=1$.

In the proof of Corollary \ref{cor38}, we made use of the fact that the distribution of $T$ is self-decomposable, and hence infinitely divisible. We can therefore immediately deduce from the representation (\ref{eq0}) of the sum $S_n$ that the distribution of $S_n$ self-decomposable, and hence infinitely divisible, a fact that was recently proved by \cite{gnp24} via an alternative constructive argument. 

\begin{corollary}\label{cor444} Let $\mu_X,\mu_Y\in\mathbb{R}$, $\sigma_X,\sigma_Y>0$, $\rho\in(-1,1)$ and $n\geq1$. Then, the distribution of the sum $S_n$ is self-decomposable. Moreover, the sum $S_n$ has the L\'evy-Khintchine representation
\begin{align}
\label{levk24}
\varphi_{S_n}(t)=\exp\bigg(\int_{-\infty}^\infty(\mathrm{e}^{\mathrm{i}tx}-1)\nu_{S_n}(x)\,\mathrm{d}x\bigg),    
\end{align}
where the L\'evy density is given by
\begin{align*}
\nu_{S_n}(x)&=\frac{n}{2}\bigg(\frac{1}{|x|}+\frac{1}{(1+\rho)^2s}\bigg(\frac{\mu_X}{\sigma_X}+\frac{\mu_Y}{\sigma_Y}\bigg)^2\bigg)\exp\bigg(-\frac{|x|}{s(1+\rho)}\bigg)\mathbf{1}_{x<0}\\
&\quad-\frac{n}{2}\bigg(\frac{1}{x}+\frac{1}{(1-\rho)^2s}\bigg(\frac{\mu_X}{\sigma_X}-\frac{\mu_Y}{\sigma_Y}\bigg)^2\bigg)\exp\bigg(-\frac{x}{s(1-\rho)}\bigg)\mathbf{1}_{x>0}.    
\end{align*}
\end{corollary}

\begin{proof}
As noted in the proof of Corollary \ref{cor38}, the distribution of  $T$ is self-decomposable (and hence infinitely divisible) for $r>0$, $\lambda_1,\lambda_2\geq0$. Moreover, the characteristic function of $T$ has the L\'evy-Khintchine representation
\begin{align}\label{levk}
\varphi_T(t)=\exp\bigg(\int_{-\infty}^\infty(\mathrm{e}^{\mathrm{i}tx}-1)\nu_T(x)\,\mathrm{d}x\bigg),    
\end{align}
where the L\'evy density is given by
\begin{equation*}
\nu_T(x)=\frac{\mathrm{e}^{-|x|/2}}{2}\bigg(\lambda_1+\frac{r}{|x|}\bigg)\mathbf{1}_{x<0}-\frac{\mathrm{e}^{-x/2}}{2}\bigg(\lambda_2+\frac{r}{x}\bigg)\mathbf{1}_{x>0},
\end{equation*}   
which follows immediately from the L\'evy-Khintchine representation of the characteristic function of the noncentral chi-square distribution that is given in \cite{j21}. One can now immediately infer that the distribution of the sum $S_n$ is self-decomposable (and hence infinitely divisible) with L\'evy-Khintchine representation (\ref{levk24}) for the characteristic function (using the L\'evy-Khintchine representation (\ref{levk}) for the characteristic function of $T$).
\end{proof}

In the following corollary, we apply the representation (\ref{eq0}) of the sum $S_n$ to obtain a new formula for the moments of $S_n$ as a finite sum of confluent hypergeometric functions of the first kind. Formulas for the first four central moments of the product $Z$ are given by \cite{h42} and formulas for the first four raw and central moments and the skewness and kurtosis of the sample mean $S_n/n$ are given by \cite{gls24}, as well as recursive formulas for higher order moments. However, as confluent hypergeometric functions can be accurately and efficiently evaluated using modern computational algebra packages and mathematical software including the GNU Scientific Library, our formula (\ref{momf3}) is the simplest to implement and most practically useful formula in the literature for computing higher order moments of the sum $S_n$. It should be noted that by using the representation (\ref{eq0}) to represent the sum $S_n$ as a difference of independent noncentral chi-square random variables we are able to obtain a simpler formula for the moments of $S_n$ than by directly working with the summation representation $S_n=\sum_{i=1}^nZ_i$.

\begin{corollary}\label{thm2} Let $\mu_X,\mu_Y\in\mathbb{R}$, $\sigma_X,\sigma_Y>0$, $\rho\in(-1,1)$ and $n\geq1$. Then, for $k\geq1$,
\begin{align}
\mathbb{E}[S_n^k]&=\frac{s^k\mathrm{e}^{-(\lambda_++\lambda_-)/2}}{(\Gamma(n/2))^2}\sum_{j=0}^k(-1)^{k-j}\binom{k}{j}(1+\rho)^j(\rho-1)^{k-j}\Gamma\bigg(j+\frac{n}{2}\bigg)\Gamma\bigg(k-j+\frac{n}{2}\bigg)\nonumber\\
&\quad\times M\bigg(j+\frac{n}{2},\frac{n}{2},\frac{\lambda_+}{2}\bigg)M\bigg(k-j+\frac{n}{2},\frac{n}{2},\frac{\lambda_-}{2}\bigg),\label{momf3}
\end{align} 
where $s=\sigma_X\sigma_Y$, and $\lambda_+$ and $\lambda_-$ are defined as in (\ref{lambda}). Also, the cumulants are given by
\begin{align}\label{37}
\kappa_k=\frac{s^k}{2}(k-1)![(1+\rho)^k(n+k\lambda_+)+(-1)^k(1-\rho)^k(n+k\lambda_-)],\quad k\geq1.
\end{align}
\end{corollary}

\begin{proof} By the representation (\ref{eq0}) and the binomial theorem we have that
\begin{equation}\label{eq2}
 \mathbb{E}[S_n^k]=\frac{s^k}{2^k}\sum_{j=0}^k(-1)^{k-j}\binom{k}{j}(1+\rho)^j(\rho-1)^{k-j}\mathbb{E}[V_1^j]\mathbb{E}[V_2^{k-j}], 
\end{equation}
and formula (\ref{momf3}) now follows from combining (\ref{eq2}) with the moment formula (\ref{momf}). The cumulant formula (\ref{37}) is obtained by using the representation (\ref{eq0}) and then proceeding similarly to how we derived the cumulant formula (\ref{cuml}).
\end{proof}

\begin{remark} Since $M(a,b,0)=1$ (see \cite[equation 13.2.2]{olver}), formula (\ref{momf3}) simplifies if $\mu_X/\sigma_X+\mu_Y/\sigma_Y=0$ or $\mu_X/\sigma_X-\mu_Y/\sigma_Y=0$ (so that $\lambda_+=0$ or $\lambda_-=0$). In the case $\mu_X=\mu_Y=0$, the moments of $S_n$ can be expressed in terms of a single hypergeometric function (see \cite[Corollary 2.4]{gaunt23}).  Similar comments apply to formula (\ref{momf2}) of Proposition \ref{prop1}.
\end{remark}




In the following corollary, we obtain an exact formula for the probability $\mathbb{P}(S_n\leq0)$. The formula is expressed in terms of the regularized incomplete beta function $I_x(a,b)=B_x(a,b)/B(a,b)$, where $B_x(a,b)=\int_0^xt^{a-1}(1-t)^{b-1}\,\mathrm{d}t$ is the incomplete beta function and $B(a,b)$ is the beta function.  

\begin{corollary}\label{corlast}  Let $\mu_X,\mu_Y\in\mathbb{R}$, $\sigma_X,\sigma_Y>0$, $\rho\in(-1,1)$ and $n\geq1$. Then
\begin{align} 
\mathbb{P}(S_n\leq0)=\mathrm{e}^{-(\lambda_++\lambda_-)/2}\sum_{j=0}^\infty\sum_{k=0}^\infty\frac{1}{j!k!}\bigg(\frac{\lambda_+}{2}\bigg)^j\bigg(\frac{\lambda_-}{2}\bigg)^kI_{\frac{1-\rho}{2}}\bigg(\frac{n}{2}+j,\frac{n}{2}+k\bigg),\label{ibf}
\end{align}
where $\lambda_+$ and $\lambda_-$ are defined as in (\ref{lambda}).
\end{corollary}

\begin{proof} From the representation (\ref{eq0}) we have that $S_n=_ds(1+\rho)V_1/2-s(1-\rho)V_2/2$, where $s=\sigma_X\sigma_Y$, and $V_1\sim\chi_n'^2(\lambda_+)$ and $V_2\sim \chi_n'^2(\lambda_-)$ are independent. Therefore 
\begin{align*}
\mathbb{P}(S_n\leq0)=\mathbb{P}\bigg(R\leq \frac{1-\rho}{1+\rho}\bigg),    
\end{align*}
where $R=V_1/V_2$. Now, the ratio $R$ follows the doubly noncentral $F$-distribution, and by equation (2.2) of \cite{b71} has cumulative distribution function
\begin{align*}F_R(x)=\mathbb{P}(R\leq x)=\mathrm{e}^{-(\lambda_++\lambda_-)/2}\sum_{j=0}^\infty\sum_{k=0}^\infty\frac{1}{j!k!}\bigg(\frac{\lambda_+}{2}\bigg)^j\bigg(\frac{\lambda_-}{2}\bigg)^kI_\frac{x}{1+x}\bigg(\frac{n}{2}+j,\frac{n}{2}+k\bigg), \quad x>0.
\end{align*}
Evaluating this expression at $x=(1-\rho)/(1+\rho)$ now yields equation (\ref{ibf}).
\end{proof}

The probability $\mathbb{P}(S_n\leq0)$, via formula (\ref{ibf}), can be efficiently computed to high precision using computational algebra packages by simply truncating the infinite series. We used \emph{Mathematica} to compute the probability $\mathbb{P}(S_n\leq0)$ for $n=1$ (that is $\mathbb{P}(Z\leq0)$) and $n=3$ and a range of parameter values. We calculated the probabilities by truncating the series at $j,k\leq50$. The results are reported in Table \ref{table1}. 

\begin{table}[h]
  \centering
  \caption{\footnotesize{$\mathbb{P}(S_n\leq0)$ for independent $Z_1,\ldots,Z_n\sim PN(\mu_X,\mu_Y,1,1,\rho)$.
  }}
\label{table1}
\footnotesize{
\begin{tabular}{l*{7}{c}}
\hline
& \multicolumn{7}{c}{$\rho$} \\
\cmidrule(lr){2-8}
$(\mu_X,\mu_Y,n)$ & $-0.75$ & $-0.50$ & $-0.25$ & 0 & 0.25 & 0.50 & 0.75 \\
\hline
        (0,0,1) & 0.7699 & 0.6667 & 0.5804 & 0.5000 & 0.4196 & 0.3333 & 0.2301 \\
        (1,$-1$,1) &  0.8636 & 0.8077 & 0.7660 & 0.7330 & 0.7075 & 0.6903 & 0.6831 \\
        (2,$-1$,1) & 0.8580 & 0.8451 & 0.8340 & 0.8258 & 0.8209 & 0.8189 & 0.8186 \\
        (2,$-2$,1) & 0.9715 & 0.9626 & 0.9579 & 0.9555 & 0.9547 & 0.9545 & 0.9545 \\
        (1,0,1) & 0.6403 & 0.5961 & 0.5483 & 0.5000 & 0.4517 & 0.4039 & 0.3597 \\ 
        (1,1,1) & 0.3169 & 0.3097 & 0.2925 & 0.2670 & 0.2340 & 0.1923 & 0.1364 \\ 
        (2,1,1) & 0.1814 & 0.1811 & 0.1791 & 0.1742 & 0.1660 & 0.1549 & 0.1420 \\
        (2,2,1) & 0.0455 & 0.0455 & 0.0453 & 0.0445 & 0.0421 & 0.0374 & 0.0285 \\ \hline
        (0,0)  &   0.9279 &  0.8045 &   0.6575 &  0.5000 &  0.3425 &  0.1955 &  0.0721 \\
(1,$-1$,3) &   0.9833 &  0.9526 &   0.9127 &  0.8652 &  0.8107 &  0.7497 &  0.6823 \\
(2,$-1$,3) &   0.9830 &  0.9712 &   0.9573 &  0.9414 &  0.9233 &  0.9031 &  0.8808 \\
(2,$-2$,3) &   0.9998 &  0.9994 &   0.9987 &  0.9979 &  0.9969 &  0.9956 &  0.6435 \\
(1,0,3)  &   0.8068 &  0.7082 &   0.6052 &  0.5000 &  0.3948 &  0.2918 &  0.1932 \\
(1,1,3)  &   0.3177 &  0.2503 &   0.1893 &  0.1348 &  0.0873 &  0.0474 &  0.0167 \\
(2,1,3)  &   0.1192 &  0.0969 &   0.0767 &  0.0586 &  0.0427 &  0.0288 &  0.0170 \\
(2,2,3)  &   0.0051 &  0.0044 &   0.0031 &  0.0021 &  0.0013 &  0.0006 &  0.0002 \\ \hline
    \end{tabular}}
\end{table}

\begin{remark} If $\mu_X/\sigma_X+\mu_Y/\sigma_Y=0$ or $\mu_X/\sigma_X-\mu_Y/\sigma_Y=0$ (so that $\lambda_+=0$ or $\lambda_-=0$), then the double infinite series (\ref{ibf}) reduces to a single infinite series. In the case $\mu_X=\mu_Y=0$ (so that $\lambda_+=\lambda_-=0$), the double sum (\ref{ibf}) reduces to a single term:
	\begin{equation*}
		\mathbb{P}(S_n\leq0)=I_{\frac{1-\rho}{2}}\bigg(\frac{n}{2},\frac{n}{2}\bigg).    
	\end{equation*}
	This formula simplifies a recent formula of \cite{gaunt24}, which expressed the probability $\mathbb{P}(S_n\leq0)$ in terms of the Gaussian hypergeometric function for the case $\mu_X=\mu_Y=0$.
	
	
\end{remark}

\section*{Acknowledgements}
The author is funded in part by EPSRC grant EP/Y008650/1 and EPSRC grant UKRI068.


\end{document}